\documentclass[english,12pt,oneside]{amsproc}
\usepackage[english]{babel}
\usepackage{a4wide}
\usepackage{amsthm}
\usepackage{graphics}
\usepackage{amsfonts, amssymb, amscd, amsmath}
\usepackage{latexsym}
\usepackage[matrix,arrow,curve]{xy}
\usepackage{mathabx}
\usepackage{color}
\usepackage{pbox}
\usepackage{tikz}
\usetikzlibrary{matrix,decorations.pathreplacing,positioning}
\usepackage{hyperref}

\DeclareMathOperator{\Cone}{Cone}

\DeclareMathOperator{\Hom}{Hom}

 \DeclareMathOperator{\free}{free}

\newcommand{\Zo}{\mathbb{Z}}
\newcommand{\Ro}{\mathbb{R}}
\newcommand{\Rg}{\mathbb{R}_{\geqslant 0}}
\newcommand{\Co}{\mathbb{C}}

\newcommand{\Talg}{\mathbb{C}^{\times}}
\newcommand{\ct}{\tilde{c}}
\newcommand{\T}{\mathcal{T}}

\newcommand{\Hr}{\widetilde{H}}
\newcommand{\dd}{\partial}

\newcommand{\F}{\mathcal{F}}

\newcommand{\CP}{\mathbb{C}P}

\newcounter{stmcounter}[section]

\numberwithin{equation}{section}

\theoremstyle{plain}
\newtheorem{cor}[stmcounter]{Corollary}

\newtheorem{thm}[stmcounter]{Theorem}

\newtheorem{prop}[stmcounter]{Proposition}
\newtheorem{lem}[stmcounter]{Lemma}

\theoremstyle{definition}
\newtheorem{defin}[stmcounter]{Definition}

\theoremstyle{remark}
\newtheorem{ex}[stmcounter]{Example}
\newtheorem{rem}[stmcounter]{Remark}
\newtheorem{con}[stmcounter]{Construction}

\begin{document}

\title{Torus actions of complexity one and their local properties}

\author{Anton Ayzenberg}
\address{Faculty of computer science, Higher School of Economics and Steklov Mathematical Institute, Moscow, Russia}
\email{ayzenberga@gmail.com}

\date{\today}
\thanks{This work is supported by the Russian Science
Foundation under grant 14-11-00414.}



\begin{abstract}
We consider an effective action of a compact $(n-1)$-torus on a
smooth $2n$-manifold with isolated fixed points. We prove that
under certain conditions the orbit space is a closed topological
manifold. In particular, this holds for certain torus actions with
disconnected stabilizers. There is a filtration of the orbit
manifold by orbit dimensions. The subset of orbits of dimensions
less than $n-1$ has a specific topology which we axiomatize in the
notion of a sponge. In many cases the original manifold can be
recovered from its orbit manifold, the sponge, and the weights of
tangent representations at fixed points.
\end{abstract}

\maketitle

\section{Introduction}\label{secIntro}

An action of a compact torus $G$ on a topological space $X$ is a
classical object of study~\cite{Br}. For a point $x\in X$ let
$G_x\subset G$ denote the stabilizer subgroup and $Gx$ the orbit
of $x$. Let $p\colon X\to X/G$ be the projection to the orbit
space. Let $S(G)$ denote the set of all closed subgroups of $G$
endowed with the lower interval topology. There is continuous map
\[
\tilde{\lambda}\colon X/G\to S(G)
\]
which maps an orbit $x\in X/G$ to the stabilizer subgroup $G_x$,
see \cite{BuchRay}.

The classical idea in the study of torus actions is the following.
It is assumed that the projection map $p\colon X\to X/G$ admits a
section. Then, given the orbit space $Q=X/G$, and the continuous
map $\tilde{\lambda}\colon Q\to S(G)$ one builds a topological
model
\[
X_{(Q,\tilde{\lambda})}=(Q\times G)/\sim
\]
which is equivariantly homeomorphic to the original space $X$. The
method of constructing model spaces was used by Davis and
Januszkiewicz \cite{DJ} for the classification of manifolds which
are now called quasitoric \cite{BPnew}. This idea traces back to
the works of Vinberg \cite{Vin}.

The method can be naturally extended to the locally standard
actions of $G\cong T^n$ on $2n$-manifolds \cite{Yo}. In this case
the projection may not admit the global section, however, it
always admits a section locally, and there exists a topological
model space of such action.

Buchstaber--Terzi\'{c} \cite{BTober,BT,BT2} introduced a theory of
$(2n,k)$-manifolds in order to study the orbit spaces of more
general torus actions and to obtain topological models for such
actions. Grassmann manifolds and flag manifolds are important
families of $(2n,k)$-manifolds. In this theory a manifold is
subdivided into strata $X_\sigma$, so that the action has the same
stabilizer $T_\sigma$ for all points of a stratum. It is essential
in the definition of $(2n,k)$-manifold that there is a convex
polytope $P^k$ and a $T^k$-equivariant generalized moment map
$X^{2n}\to P^k$. Every stratum $X_\sigma$ is then represented as a
principal $T/T_x$-bundle over the product $P_\sigma^\circ\times
M_\sigma$, where $P_\sigma$ is a certain subpolytope of $P$ and
$M_\sigma$ is an auxiliary space of dimension $2(n-k)$ called the
space of parameters. Therefore the orbit space $X^{2n}/T^k$ is
represented as the union $\bigsqcup_\sigma P_\sigma\times
M_\sigma$. The theory of $(2n,k)$-manifolds provided the specific
methods to describe the topology of this union.

Whenever a compact $k$-torus acts effectively on a $2n$-manifold
we call the number $n-k$ the \emph{complexity of the action}.
While actions of complexity zero are well studied, the actions of
positive complexity constitute a harder problem. It is generally
assumed that the actions of complexity $\geqslant 2$ are extremely
complicated in general. The actions of complexity one take an
intermediate position: they were studied from several different
viewpoints. Algebraical theory of complexity one actions was
developed in the works of many authors, in particular, the
classification of such actions even in nonabelean case was given
by Timash\"{e}v \cite{Tim,Tim2}. Hamiltonian complexity one
actions on symplectic manifolds are also well studied: see e.g.
the work of Karshon--Tolman \cite{KT} and references therein.
Circle action on a $4$-manifold is a classical subject, see e.g.
\cite{ChL,Fint,OR}.

In this paper we study complexity one actions from the topological
viewpoint. Our approach is different from the one used in
\cite{BuchRay} and \cite{BT}. Instead of trying to stratify the
manifold so that the action on each stratum admits a section, we
partition the manifold by orbit types. Under two restrictions we
prove that the orbit space $Q=X/T$ is a topological manifold, see
Theorem \ref{thmManifold} for the precise statement. Note that for
this result it is not required that the stabilizers of the action
are connected. Such restriction was imposed in the theory of
$(2n,k)$-manifolds, however there exist natural examples of the
actions which have finite stabilizers but still the orbit space is
a manifold.

We make a remark on the main difference from situations considered
in toric topology: the typical action of complexity one does not
admit a section, even locally.  

Natural examples of complexity one actions which we keep in mind
are the following.
\begin{enumerate}
\item The $T^3$ action on the complex Grassmann manifold
$G_{4,2}$.
\item The $T^2$ action on the manifold $F_3$ of full complex flags in
$\Co^3$.
\item Quasitoric manifolds $X_{(P,\Lambda)}^{2n}$ with the induced action of
a generic subtorus $T\subset G$, $\dim T=n-1$.
\item The space of isospectral periodic tridiagonal Hermitian
matrices of size $n\geqslant 3$.
\end{enumerate}
Using the theory of $(2n,k)$-manifolds, Buchstaber--Terzi\'{c}
proved that the orbit space of the Grassmann manifold $G_{4,2}$ is
$S^5$, and the orbit space of the flag manifold $F_3$ is $S^4$.
These two examples motivated our study. In
Theorem~\ref{thmQToricReduc} we prove that the orbit space of a
quasitoric manifold by the action of $T$ is also homeomorphic to a
sphere $S^{n+1}$.

A space of isospectral tridiagonal $n\times n$-matrices is a more
interesting object. This space will be studied in details in the
subsequent paper \cite{AyzMatr}. This space depends on the
spectrum and for some degenerate spectra it is not smooth.
However, if it is a smooth manifold, we will prove that its orbit
space is $S^4\times T^{n-3}$. In \cite{AyzMatr} we describe
non-free part of the torus action using the regular permutohedral
tiling of the space. This allows to understand the topology of the
whole space, not just its orbit space.

The study of the space of periodic tridiagonal matrices raised
several questions about actions of complexity one. One of the
questions is the topological classification of such actions. In
this paper we prove that under certain restrictions the space $X$
with complexity one action is determined by the orbit manifold
$Q=X/T$, the set of non-free orbits $Z\subset Q$, and the weights
of tangent representations at fixed points. See Theorem
\ref{thmClassSphere} and Proposition~\ref{propSpecialHomeo}. The
set of non-free orbits have a specific topology which we
axiomatized in the notion of \emph{sponge}. Sponges seem to be the
objects of independent interest.

\section{Appropriate actions of complexity
one}\label{secGeneralDefins}

In the following, $T$ usually denotes the compact torus of
dimension $n-1$ and $G$ denotes compact tori of other dimensions.
We refer to the classical monograph of Bredon \cite{Br} for
general information of group actions on manifolds.

Let us specify the type of actions to be considered in the paper.
For a smooth action of $G$ on a smooth manifold $X$ define the
\emph{fine partition} on $X$ by orbit types
\[
X=\bigsqcup_{H\in S(G)}X^H.
\]
Here $H$ runs over all closed subgroups of $G$ and
$X^H=\tilde{\lambda}^{-1}(H)=\{x\in X\mid G_x=H\}$.

\begin{defin}
An effective action of $G$ on a compact smooth manifold $X$ is
called \emph{appropriate} if
\begin{itemize}
\item the fixed points set $X^G$ is finite;
\item (adjoining condition) the closure of every connected component of a
partition element $X^H$, $H\neq G$, contains a point $x'$ with
$\dim G_{x'}>\dim H$.
\end{itemize}

If, moreover, the stabilizer subgroup of every point is a torus,
we call the action \emph{strictly appropriate}.
\end{defin}

\begin{rem}
The adjoining condition implies that whenever a subset $X^H$ is
closed in the topology of $X$, then it is the fixed point set
$X^G$.
\end{rem}

\begin{rem}
A subgroup $H$ of a torus has the form $H_t\times H_f$, where
$H_t$ is a torus and $H_f$ is a finite abelian group. For strictly
appropriate actions the finite components $H_f$ of all stabilizers
vanish. In other words, a strictly appropriate action is an action
with all stabilizers being connected.
\end{rem}

\begin{ex}
Let an algebraical torus $(\Talg)^k$ act algebraically on a smooth
variety $X$ with finitely many fixed points. Then the induced
action of a compact subtorus $T^k\subset (\Talg)^k$ on $X$ is
appropriate, as follows from Bialynicki-Birula method \cite{BB}.
Indeed, for a given point $x\in X\setminus X^T$ consider the
1-dimensional algebraical torus $\Talg\subset (\Talg)^k$ which
acts on $x$ nontrivially. Consider the point $x'=\lim_{t\to 0}tx$,
where $0<t\leqslant 1$, $t\in \Co^\times$. The point $x'$ is
connected with $x$ and has stabilizer of bigger dimension (since
$x$ is stabilized by $(\Talg)^k_x$ as well as by $\Co^\times$).
Iterating this procedure, we arrive at some fixed point.

In particular, the action of a compact torus on a complex
GKM-manifold (see \cite{GKM}) is appropriate.
\end{ex}

\begin{ex}
The effective action of $T^{n-1}$ on $F_n$, the manifold of
complete complex flags in $\Co^n$ is strictly appropriate. The
effective action of $T^{n-1}$ on a Grassmann manifold $G_{n,k}$ of
complex $k$-planes in $\Co^n$ is also strictly appropriate.
\end{ex}

\begin{ex}
Let the action of $G\cong T^n$ on a smooth manifold $X^{2n}$ be
locally standard (see definition in Section
\ref{secLocStandActions}). The orbit space $P=X^{2n}/G$ is a
manifold with corners. This action is appropriate whenever every
face of $P$ contains a vertex. If it is appropriate, then it is
strictly appropriate. In particular, quasitoric manifolds provide
examples of strictly appropriate torus actions.
\end{ex}

\begin{ex}
Let the action of $G$ on $X$ be appropriate, and the induced
action of a subtorus $T\subset G$ on $X$ has the same fixed points
set. Then the action of $T$ is also appropriate. Indeed, the
partition element $(X')^K$ of the $T$-action for $K\subseteq T$
have the form
\[
(X')^K=\bigcup_{H\subseteq G, H\cap T=K} X^H.
\]
Therefore the adjoining condition for $G$-action implies the
adjoining condition for the induced $T$-action.
\end{ex}

Now we restrict to actions of complexity one, that is to the case
$\dim T=n-1$, $\dim X=2n$. Let $x\in X^T$ be a fixed point, and
\[
\alpha_1,\ldots,\alpha_n\in N = \Hom(T,S^1)\cong \Zo^{n-1}
\]
be the weights of the tangent representation at $x$. This means,
\[
T_xX\cong V(\alpha_1)\oplus\cdots\oplus V(\alpha_n),
\]
where $V(\alpha)$ is the standard 1-dimensional complex
representation given by
\[
tz=\alpha(t)\cdot z,\quad z\in \Co.
\]
If there is no complex structure on $X$, then we have an ambiguity
in choice of signs of $\alpha_i$. These signs do not affect the
following definitions.

\begin{defin}
A representation of $T^{n-1}$ on $\Co^n$ is called in general
position if every $n-1$ of its $n$ weights are linearly
independent. An action of $T=T^{n-1}$ on $X=X^{2n}$ is called
\emph{an action in general position} if its tangent representation
at any fixed point is in general position.
\end{defin}

\begin{rem}
For a given $n$-tuple of weights $\alpha_1,\ldots,\alpha_n$ there
is a relation $c_1\alpha_1+\cdots+c_n\alpha_n=0$ in $N\cong
\Zo^{n-1}$. 
The action is in general position if $c_i\neq 0$ for
$i=1,\ldots,n$.
\end{rem}

\begin{thm}\label{thmManifold}
Consider an appropriate action of $T=T^{n-1}$ on $X=X^{2n}$ and
assume it is in general position. Then the orbit space $Q=X/T$ is
a topological manifold.
\end{thm}

\begin{proof}
First we prove the local statement near fixed points.

\begin{lem}\label{lemLocalQuotient}
For a representation of $T=T^{n-1}$ on $\Co^n$ in general position
we have $\Co^n/T\cong \Ro^{n+1}$.
\end{lem}

\begin{proof}
Consider the standard action of $G=T^n$ on $\Co^n$ which rotates
the coordinates. The weights of the standard action
$e_1,\ldots,e_n$ is the standard basis of the character lattice
$\Hom(G,S^1)\cong \Zo^n$. Consider the lattice homomorphism
$\phi\colon \Zo^n\to N$ given by $\phi(e_i)=\alpha_i$,
$i=1,\ldots,n$. This homomorphism is induced by some homomorphism
$\phi^*\colon T\to G$ of tori. The given action of $T$ is the
composition of $\phi^*$ with the standard action.

So far we may assume that there is an action of a subtorus
$T'=f(T)\subset G$ where $G$ acts in a standard way. The torus
$T'$ is given by $\{t_1^{c_1}\cdots t_n^{c_n}=1\}$, where
$(c_1,\ldots,c_n)$ is a linear relation on the weights $\alpha_i$
and $\gcd\{c_i\}=1$. The condition of general position implies
that all $c_i\neq 0$. Hence the intersection of $T'$ with each
coordinate circle in $G$ is a finite subgroup.

Let us denote the space $\Co^n/T=\Co^n/T'$ by $Q$. We have the map
$g\colon Q\to \Co^n/G\cong \Rg^n$, which sends $T$-orbit to its
$G$-orbit. For every $p\in \Ro_>^n$ the preimage $g^{-1}(p)$ is a
circle $G/T'$. For every $p\in \dd\Rg^n$, the preimage $g^{-1}(p)$
is a single point, since the product of $T'$ with any nontrivial
coordinate subtorus generate the whole torus $G$. Therefore we
have $Q=\Rg^n\times S^1/\sim$, where $\sim$ collapses circles over
$\dd\Rg^n$. We have
\[
(\Rg^n\times S^1/\sim)\cong(\Ro^{n-1}\times\Rg\times
S^1)/\sim\cong \Ro^{n-1}\times\Co.
\]
which proves the lemma.
\end{proof}

%

%
%

We now prove the theorem by induction on the dimension of
stabilizer subgroup. If $\dim H=n-1$, that is $H=T$, Lemma
\ref{lemLocalQuotient} shows that $X/T$ is a manifold near the
fixed point set $X^T/T$. Now let $[x]\in X/T$ be an orbit such
that $T_x=H$, that is $x\in X^H$. Due to the adjoining condition,
there exists a point $x'$ such that the local representations at
$x$ and $x'$ coincide and $x'$ is close to a partition element
$X^{H'}$ with $\dim H'>\dim H$. Here by the local representation
we mean a representation of $T_x$ on the normal space
$T_xX/T_xT(x)$ to the orbit.

By induction, the space $X/T$ is a manifold near $X^{H'}/T$
therefore there exists a neighborhood of $[x']$ homeomorphic to
$\Ro^{n+1}$. Therefore there exists a neighborhood of $[x]$
homeomorphic to $\Ro^{n+1}$. Indeed, both neighborhoods are
homeomorphic to the orbit space of the local representation
according to slice theorem.
\end{proof}

\begin{rem}
Let $v_1,\ldots,v_{n-1}\in\Ro^{n-1}$ be the basis of a vector
space and $v_n=-v_1-\cdots-v_{n-1}$. Consider the subset $C$ of
$\Ro^{n-1}$ given by
\[
C=\bigcup_{I\subset [n],|I|=n-2} \Cone(v_i\mid i\in I).
\]
This subset is homeomorphic to the $(n-2)$-skeleton of the
standard nonnegative cone $\Rg^n$.

The subset $C$ is the $(n-2)$-skeleton of the simplicial fan
$\Delta_{n-1}$ of type $A_{n-1}$; it comes equipped with the
natural filtration
\[
C_0\subset\cdots\subset C_{n-2}=C
\]
where $C_k$ is the union of cones of dimension $k$ of
$\Delta_{n-1}$. This filtration can be defined topologically: we
say that $x\in \Ro^{n-1}$ has type $k$ if $C$ cuts a small disc
$B_x$ around $x$ into $n-k$ chambers. Then $C_k$ consists of all
points of type $\leqslant k$.
\end{rem}

Next we introduce a notion of the subspace in a topological
manifold which is locally modeled by the subset $C\subset
\Ro^{n-1}\subset \Ro^{n+1}$. Assume we are given a topological
manifold $Q$ of dimension $n+1$ and a subset $Z\subset Q$.

\begin{defin}
A subset $Z\subset Q$ is called a \emph{sponge} if, for any point
$x\in Z$, there is a neighborhood $U_x\subset Q$ such that
$(U_x,U_x\cap Z)$ is homeomorphic to $(V\times \Ro^2,(V\cap
C)\times\Ro^2)$, where $V$ is an open subset of the space
$\Ro^{n-1}\supset C$.
\end{defin}

Every sponge is filtered in a natural way compatible with the
filtration of $C$. We say that a point $x\in Z\subset Q$ has type
$k$ if $H^2(U_x\setminus Z;\Zo)\cong \Zo^{n-k-1}$ for a small disc
neighborhood $x\in U_x\subset Q$. Then $Z_k$ consists of all
points of type at most $k$. Note that $\dim Z_k=k$. Informally
speaking, the sponge set is a collection of $(n-2)$-manifolds with
corners, and the corners are stacked together like maximal cones
in $C$. The case $n=4$ is shown on Fig.\ref{figSponge}

\begin{figure}[h]
\begin{center}
\includegraphics[scale=0.2]{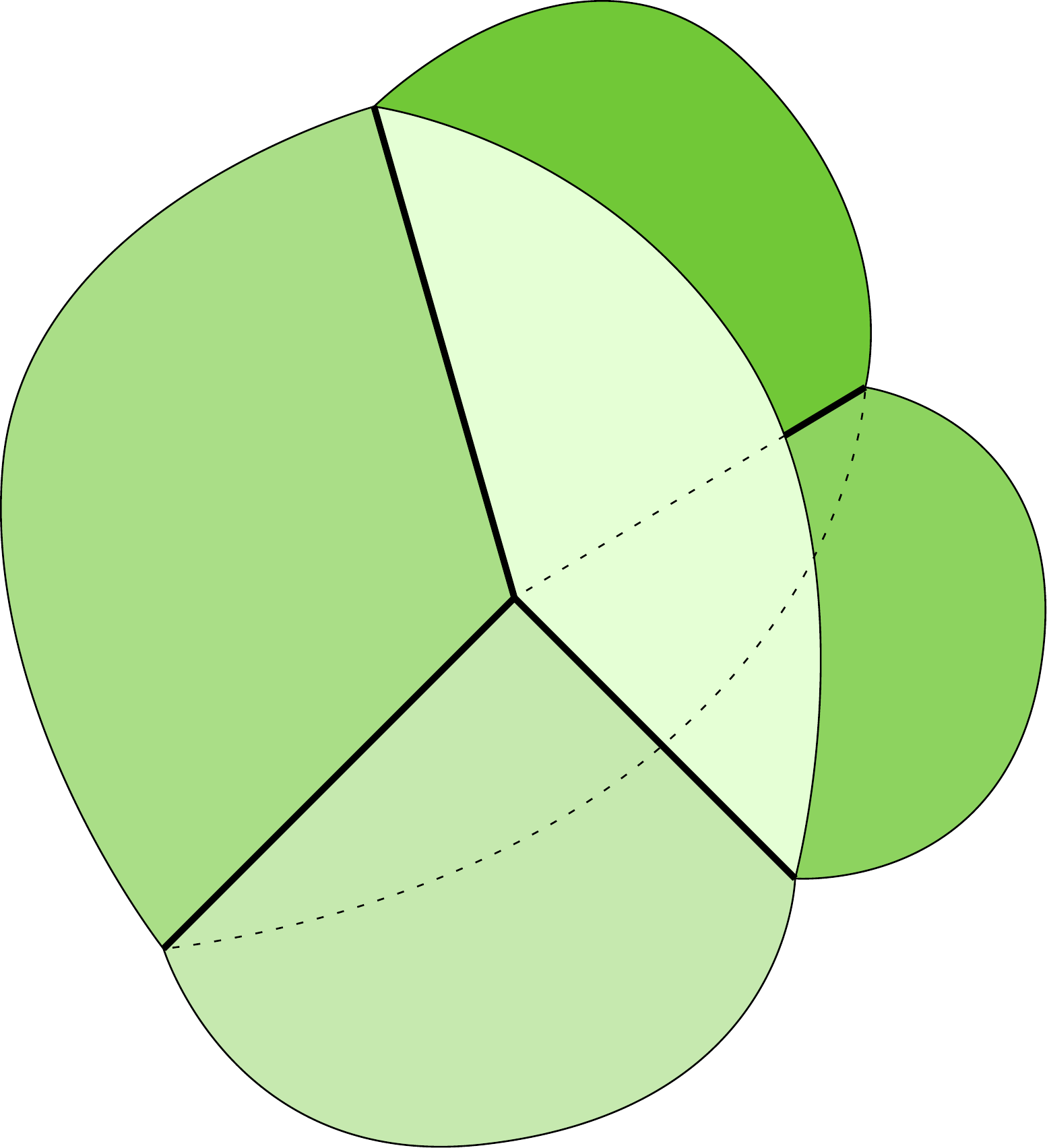}
\end{center}
\caption{Local structure of the sponge for
$n=4$.}\label{figSponge}
\end{figure}

\begin{con}
For a general action of a torus $G$, $\dim G=m$ on $X$ we can
consider the \emph{coarse filtration}:
\[
X_0\subset X_1\subset\cdots\subset X_m=X
\]
where $X_i=\bigcup_{\dim H\geqslant m-i}X^H$ is the union of all
orbits of dimension at most $i$. In particular, the set
$X\setminus X_{m-1}$ is the locus of almost free action (``almost
free'' action means ``have only finite stabilizers''). There is an
induced coarse filtration on $Q=X/G$:
\[
Q_0\subset Q_1\subset\cdots\subset Q_m
\]
\end{con}

\begin{rem}
The terms ``fine partition'' and ``coarse filtration'' refer to
the following fact. The fine partition distinguishes different
subgroups of the torus. However, coarse filtration distinguishes
only the dimensions of the subgroups.
\end{rem}

\begin{prop}
For an appropriate action in general position of $T^{n-1}$ on
$X^{2n}$ we get a topological manifold $Q=X/T$. The coarse
filtration on $Q$ has the form
\[
Z_0\subset Z_1\subset\cdots\subset Z_{n-2}=Z\subset Q
\]
where $Z\subset Q$ is a sponge. The filtration by orbit dimensions
coincides with the sponge filtration defined topologically.
\end{prop}

\begin{proof}
The local statement near fixed points is proved in Lemma
\ref{lemLocalQuotient}. The global case follows by the adjoining
condition similar to the proof of Theorem \ref{thmManifold}.
\end{proof}

\section{Characteristic data}\label{secCharData}

Assume there is an appropriate action of $T=T^{n-1}$ in general
position on $X=X^{2n}$. We allow $X$ to have a boundary, however
in this case we require that the action is free on $\dd X$. The
same arguments as before show that $Q=X/T$ is a topological
manifold with boundary and its boundary $\dd Q$ is $\dd X/T$.

In this section we assume that the actions are strictly
appropriate. This means that there are no finite components in
stabilizer subgroups and therefore, the face partition of a coarse
filtration coincides with the fine partition on $Q$. With a
strictly appropriate action in general position we assign the
characteristic data $(Q,Z,\mu,e)$ consisting of the following
information
\begin{itemize}
\item $Q=X/T$, the orbit space.
\item $Z\subset Q^\circ$, the sponge subset determined by the
action:
\[
Z_0\subset Z_1\subset\cdots\subset Z_{n-2}=Z,
\]
where $Z_i\subset Q$ is the set of orbits of dimension at most
$i$. The closure of a connected component of $Z_i\setminus
Z_{i-1}$ is called a \emph{face} of $Z$ of dimension $i$. Faces of
dimension $n-2$ are called \emph{facets}. Every face of dimension
$i$ is contained in exactly ${n-i\choose 2}$ different facets. The
stabilizer remains the same for all points in the interior of any
given face $F$, since no finite components are allowed in the
stabilizers. This stabilizer will be denoted $T_F$. Let
$\F=\{F_1,\ldots,F_m\}$ be the set of facets of $Z$.

\item $\mu$ is a \emph{characteristic map}
\[
\mu\colon \F\to\{1\mbox{-dimensional subgroups of
}T^{n-1}\}=\Hom(S^1,T^{n-1})\cong \Zo^{n-1}
\]
It sends a facet $F_k$ to the oriented stabilizer subgroup
$T_{F_k}$ of any of its interior points (we introduce orientation
arbitrarily, see details in Section \ref{secOrientations}). For
any face $F$ of dimension $i$ we have $F=\bigcap_{i\in I}F_i$ for
certain subset $I\subset [m]$, $|I|={n-i\choose 2}$. The
stabilizer $T_F$ is the product of $\mu(F_i)=T_{F_i}$ inside the
torus $T^{n-1}$. Note that this product is generally not free,
since it has dimension $n-1-i$. However, it can be seen that
characteristic map $\mu$ determines the stabilizers of all points.

\item $e\in H^2(Q\setminus Z; H^2(BT))$ is the \emph{Euler class}
of the free part of the action. In details, every orbit in
$Q\setminus Z$ is full-dimensional and there are no finite
stabilizer subgroups, therefore the free part of the action is a
principal $T$-bundle $p\colon X^{\free}\to Q\setminus Z$. This
bundle is classified by the homotopy class of a map
\[
Q\setminus Z\to BT\cong (\CP^\infty)^{n-1}\cong K(\Zo^{n-1};2).
\]
Such maps also classify the second cohomology group of $Q\setminus
Z$. Therefore, we have the classifying element
\[
e\in H^2(Q\setminus Z;\Zo^{n-1}),\mbox{ where }\Zo^{n-1}\cong
H_2(BT;\Zo)\cong H_1(T;\Zo).
\]
\end{itemize}

\begin{rem}
Note that unlike the half-dimensional torus actions the
characteristic data of complexity one actions can not be
arbitrary. It will be shown in this and the next section that the
Euler class $e$ and the characteristic function $\mu$ determine
each other to much extent. Moreover, the Euler class of complexity
one actions is always nontrivial.
\end{rem}

Let $x\in Z\subset Q$ be a point of type $k\leqslant n-2$. Let
$U_x$ be a small neighborhood of $x$ in $Q$, homeomorphic to an
open disc. Let $i_x\colon U_x\to Q$ be the inclusion map. We have
an induced homomorphism
\[
H^2(Q\setminus Z;H_1(T))\to H^2(U_x\setminus Z;H_1(T))
\]
The image of $e\in H^2(Q\setminus Z;H_1(T))$ by this homomorphism
is denoted
\[
e_x\in H^2(U_x\setminus Z;H_1(T))\cong \Zo^{n-k-1}\otimes H_1(T)
\]
and called the local Euler class at $x$. Recall that the type of
the point is defined by the rank of the second cohomology of
$U_x\setminus Z$, see section \ref{secGeneralDefins}.

In particular, if $x$ has type $n-2$ (i.e. $x$ lies in the
interior of a facet), the neighborhood $U$ can be chosen in a way
that $U_x\cap Z\cong \Ro^{n-2}$. In this case we have
$U_x\setminus Z\cong \Ro^{n+1}\setminus \Ro^{n-2}$ and
\[
H^2(U_x\setminus Z;H_1(T))\cong H^2(\Ro^{n+1}\setminus
\Ro^{n-2};H_1(T))\cong H^1(T).
\]
The last isomorphism is canonical provided $Q$ (hence $U_x$) is
oriented.

\begin{defin}
The Euler class $e$ and characteristic function $\mu$ are called
\emph{compatible} if the following condition holds: for any $x\in
Z$, the map $H_1(T)\to H_1(T/T_x)$ induced by the quotient map
$T\to T/T_x$ sends $e_x\in H^1(T)$ to zero.
\end{defin}

\begin{prop}
Assume there is an appropriate action in general position of
$T=T^{n-1}$ on a manifold $X=X^{2n}$. Then its characteristic data
$e$ and $\mu$ are compatible.
\end{prop}

\begin{proof}
As before, let $Q=X/T$ be the orbit space, $Z\subset Q$ the set of
orbits of dimensions $\leqslant n-2$, and $p\colon X\to Q$ the
projection map. Take any point $x\in Z\subset Q$. We can choose a
small neighborhood $U_x\ni x$, $U_x\subset Q$ such that
stabilizers of any point $y\in U_x$ are contained in $T_x$ and
$U_x\cong \Ro^{n+1}$. Consider the map
\[
f\colon p^{-1}(U_x)/T_x\stackrel{T/T_x}{\longrightarrow} U_x
\]
taking the remaining quotient. Since all stabilizers of points in
$U_x$ are contained in $T_x$, the map $f$ is a principal
$T/T_x$-bundle. It is a trivial bundle since $U_x$ is
contractible, therefore the induced $T/T_x$-bundle over
$U_x\setminus Z$ is also trivial. Hence its Euler class vanishes.
However, this Euler class is the image of $e_x\in H^2(U_x\setminus
Z;H_1(T))$ under the induced map $H_1(T)\to H_1(T/T_x)$, which
proves the statement.
\end{proof}

\begin{rem}
For a point $x$ in the interior of a facet $F_i$ the stabilizer
$T_x$ is one-dimensional. In this case the compatibility condition
tells that $e_x$ is proportional to the fundamental class of
$T_x=\mu(F_i)$:
\[
e_x=k_i\mu(F_i)\in H_1(T;\Zo)\cong \Hom(S^1,T).
\]
The constants $k_i\in \Zo$ can be determined from the weights of
the tangent representation at any fixed point adjacent to $F_i$.
It will be shown in Section \ref{secOrientations} that all these
constants are actually $\pm1$ for strictly appropriate actions.
\end{rem}

\begin{con}\label{constrModelSpace}
Let us construct a topological model space given abstract
compatible characteristic data. Assume a topological
$(n+1)$-manifold $Q$ is given, and let $Z\subset Q$ be a sponge
with facets $F_1,\ldots,F_m$. Let $\mu$ be a map assigning a
1-dimensional subgroup of $T=T^{n-1}$ to any facet $F_i$ with the
following property: if a $k$-dimensional face $F$ of a sponge lies
in facets $F_{i}$ with labels $i\in I$, $|I|={n-k\choose 2}$, then
\[
\dim\prod_{i\in I}\mu(F_i)=k.
\]
The subgroup $\prod_{i\in I}\mu(F_i)$ will be denoted $T_x$ if $x$
lies in interior of $F$. If $x\in Q\setminus Z$ we set
$T_x=\{1\}\subset T$. Finally, fix a class $e\in H^2(Q\setminus Z;
H_1(T))$ compatible with $\mu$. With all this information fixed,
introduce a space $Y=Y_{(Q,Z,\mu,e)}$. As a set,
\[
Y=\bigsqcup_{x\in Q} T/T_x.
\]
The topology is introduced in two steps.

(1) The topology on a subset
\[
Y^{\free}=\bigsqcup_{x\in Q\setminus Z} T/T_x= \bigsqcup_{x\in
Q\setminus Z} T\subset Y
\]
is introduced in a way such that the natural projection
$Y^{\free}\to Q\setminus Z$ is the principal $T$-bundle classified
by $e\in H^2(Q\setminus Z;H_1(T))$.

(2) For a point $y$ in $\bigsqcup_{x\in Z} T/T_x$ we specify the
basis of topology. Let $x\in Z$ and $t_x\in T/T_x\subset Y$. To
define the base of topology near $t_x$, we fix a small open
neighborhood $U_x\subset Q$ of $x$ and for each $x'\in U_x$ take a
projection of tori $p_{x'}\colon T/T_{x'}\to T/T_x$. This is well
defined since $U_x$ is assumed small enough so that $T_x$ contains
any other stabilizer $T_{x'}$. Let $V$ be a neighborhood of $t_x$
in $T/T_x$. The subsets of the form
\[
\bigsqcup_{x'\in U_x} p_{x'}^{-1}(V)
\]
form the base of topology around $t_x$. Note that since $e$ and
$\mu$ are compatible, we have a trivial principal $T/T_x$-bundle
over $A\to U_x\setminus Z$ therefore the topology defined in (2)
is compatible with the one defined in (1) on a subset
$U_x\setminus Z$.

Finally, define the $T$-action on each fiber $T/T_x$ as given by
the projection $T\to T_x$. It can be seen that $Y$ is a compact
Hausdorff topological space carrying the continuous action of $T$.
Its orbit space is homeomorphic to $Q$.
\end{con}

The constructed space $Y=Y_{(Q,Z,\mu,e)}$ is not necessarily a
manifold.

\begin{ex}
Assume $e_x=0$ for some point $x$ lying in interior of a facet
$F_j$. Then $Y$ is not a manifold over $x$. Indeed, by
construction, a neighborhood of $x$ in $Y$ is homeomorphic to
$U_x\times T/\sim$, where $(x',t')\sim(x'',t'')$ whenever
$x'=x''\in Z$ and $t'(t'')^{-1}\in \mu(F_j)$. This subset is not a
manifold, which can be shown by computing its local homology
groups for points lying over $Z$.
\end{ex}

\begin{prop}\label{propHomeoToModel}
Let $X=X^{2n}$ be a manifold with strictly appropriate action of
$T=T^{n-1}$ in general position. Let $(Q,Z,\mu,e)$ be its
characteristic data. Let $Y$ be the model space constructed from
the data $(Q,Z,\mu,e)$. Then there is a $T$-equivariant
homeomorphism $h\colon Y\to X$ which induces the identity
homeomorphism on the orbit space $Q$:
\[
\xymatrix{ Y\ar[r] \ar[d]^{p_Y}& X \ar[d]^{p_X} \\
Q \ar@{=}[r]& Q}
\]
\end{prop}

\begin{proof}
The equivariant homeomorphism over $Q\setminus Z$ follows
immediately, since both $p_X^{-1}(Q\setminus Z)$ and
$p_Y^{-1}(Q\setminus Z)$ are the principal $T$-bundles classified
by $e$. For a point $x\in Z\subset Q$, the equivariant
homeomorphism $h\colon p_Y^{-1}(U_x\setminus Z)\to
p_X^{-1}(U_x\setminus Z)$ is extended uniquely to the equivariant
homeomorphism $h\colon p_Y^{-1}(U_x)\to p_X^{-1}(U_x)$. Indeed,
there is a unique equivariant homeomorphism $\tilde{h}\colon
p_Y^{-1}(U_x)/T_x\to p_X^{-1}(U_x)/T_x$ which extends the
homeomorphism $h/T_x\colon p_Y^{-1}(U_x\setminus Z)/T_x\to
p_X^{-1}(U_x\setminus Z)/T_x$, since both spaces are trivial
$(T/T_x)$-bundles over $U_x$ (according to compatibility
condition) and $U_x\setminus Z$ is dense in $U_x$. For a point
$t_x\in T/T_x\subset p_Y^{-1}(U_x)$ over $x$ there is a unique
point $\alpha\in p_X^{-1}(U_x)$ such that
$\tilde{h}([t_x])=[\alpha]$, since the projection map
$p_X^{-1}(U_x)\to p_X^{-1}(U_x)/T_x$ is a bijection over $x$.
Hence we can extend $h$ by putting $h(t_x)=\alpha$.

This procedure defines an equivariant continuous bijection between
compact spaces $Y$ and $X$. Since $X$ is compact and $Y$ is
Hausdorff it is an equivariant homeomorphism.
\end{proof}

\section{Orientation issues and details}\label{secOrientations}

Consider a representation of $T=T^{n-1}$ on $\Co^n$ in general
position. The weights $\alpha_1,\ldots,\alpha_n\in \Hom(T,S^1)$
are defined up to sign.

\begin{defin}
An \emph{omniorientation} is a choice of the orientation of $T$
(hence the orientation of the lattice $N=\Hom(T,S^1)$) and the
choice of signs of all vectors $\alpha_i$. 
\end{defin}

\begin{con}
Assume there is a fixed basis in the lattice $N$, so that
$\alpha_j$ is written in coordinates
$\alpha_j=(\alpha_{j,1},\ldots,\alpha_{j,n-1})$. For each
$i=1,\ldots,n$ consider the determinant of the matrix formed by
$\alpha_j$ with $j\neq i$:
\[
\ct_i=(-1)^i\alpha_1\wedge\cdots\wedge
\widehat{\alpha_i}\wedge\cdots\wedge\alpha_n\in
\Lambda^{n-1}N\cong\Zo
\]
Since $\alpha_i$ are in general position we have $\ct_i\neq 0$ for
all $i=1,\ldots,n$. Cramer's rule implies
\[
\ct_1\alpha_1+\cdots+\ct_n\alpha_n=0.
\]
Let $c_{\gcd}=\gcd(\ct_1,\ldots,\ct_n)$ and $c_i=\ct_i/c_{\gcd}$.
Let $G=T^n$ act on $\Co^n$ in a standard way
\[
(t_1,\ldots,t_n)\cdot (z_1,\ldots,z_n)=(t_1z_1,\ldots,t_nz_n)
\]
and let $T$ be a subtorus
\begin{equation}\label{eqTprime}
T'=\{t_1^{c_1}\cdots t_n^{c_n}=1\}\subset G
\end{equation}
The proof of Lemma \ref{lemLocalQuotient} implies that the orbit
space of the representation of $T$ on $\Co^n$ coincides with the
orbit space of the induced action of $T'$ on $\Co^n$, therefore we
may not distinguish these two cases.
\end{con}

\begin{lem}\label{lemPMone}
The representation action of $T=T'$ on $\Co^n$ in general position
is strictly appropriate if and only if $c_i=\pm1$, that is all
parameters $\ct_i$ coincide up to sign.
\end{lem}

\begin{proof}
The point $(0,\ldots,0,1,0,\ldots,0)$ with unit at $j$-th position
has the stabilizer $T'\cap G_j$, where $T'$ is given by
\eqref{eqTprime} and $G_j$ is the $j$-th coordinate circle of
$G\cong T^n$. This stabilizer is isomorphic to the cyclic group
$\Zo_{c_j}$. If the action is strictly appropriate, then there are
no finite components in stabilizer subgroups, so far $c_i$ is
necessarily $\pm1$.

The converse statement is similar. The stabilizers of $T'$-action
on $\Co^n$ have the form $T'\cap G_I$ for all possible coordinate
subtori $G_I\in G$, $I\subseteq [n]$. This group has finite
component of order $\gcd(c_i\mid i\in I)$. Hence, if all $c_i$ are
$\pm1$, these finite components vanish.
\end{proof}

Recall that $C\subset \Ro^{n-1}\subset\Ro^{n+1}$ denotes the
$(n-2)$-skeleton of the fan of type $A_{n-1}$. This space is the
sponge of an appropriate representation action of $T$ on $\Co^n$.

In the following we only consider strictly appropriate actions.
The facets $\{F_{i,j}\mid 1\leqslant i<j\leqslant n\}$ of $Z$ are
labeled in a way that $F_{i,j}$ is ``spanned'' by all weights
except $\alpha_i$ and $\alpha_j$. Let us fix an orientation on
1-dimensional stabilizers of the action (this corresponds to the
choice of the signs of the characteristic values $\mu(F_{i,j})\in
\Hom(S^1,T)$). These orientations determine the orientation of the
orbit $Tx\cong T/\mu(F_{i,j})$ for $x\in F_{i,j}^\circ$. The
preimage of $F_{i,j}^\circ$ under the projection map has the form
$\{(z_1,\ldots,z_n)\in\Co^n\mid z_i=z_j=0,z_k\neq 0$ for $k\neq
i,j\}$, this space has a canonical orientation determined by the
complex structure on $\Co^n$. Therefore the orientations of the
stabilizer circles determine the orientations of facets $F_{i,j}$.
Finally, since the orientation on $\Co^n/T\cong \Ro^{n+1}$ is
fixed, the orientation of $F_{i,j}$ determines the orientation of
a small 2-sphere $S_{i,j}^2$ around $F_{i,j}$. Let us describe the
Euler class of the free part of action.

\begin{prop}\label{propPMone}
The Euler class $e\in H^2(Q\setminus Z;H_1(T))$ of a strictly
appropriate representation action of $T$ on $\Co^n$ is given by
the condition
\[
\langle e, [S_{i,j}^2]\rangle = \frac{c_i}{c_j}\mu(F_{i,j})\in
H_1(T)\cong \Hom(S^1,T),
\]
for a small 2-sphere around facet $F_{i,j}$, $1\leqslant
i<j\leqslant n$.
\end{prop}

The constants $c_i$ were defined earlier in this section. Lemma
\ref{lemPMone} shows that for strictly appropriate actions
$c_i=\pm1$. Note that $\frac{c_i}{c_j}=\frac{\ct_i}{\ct_j}$, and
parameters $\ct_i,\ct_j$ can be computed from the weight vectors.

\begin{proof}
Assume $i=1$, $j=2$ for simplicity. The preimage of a sphere
$S_{1,2}^2$ in the space $\Co^n$ has the form
\[
M=\{(z_1,\ldots,z_n)\in\Co^n\mid
|z_1|^2+|z_2|^2=\varepsilon,|z_k|=\varepsilon, k>2\}\mbox{ for
small }\varepsilon>0.
\]
The subtorus $T=\{t_1^{c_1}\cdots t_n^{c_n}=1\}\subset G$ acts
freely on $M$. The stabilizer $T_x=\mu(F_{1,2})$ for $x\in
F_{1,2}^{\circ}$ has the form
\[
T_x=\{t_1^{c_1}t_2^{c_2}=1,t_k=1,k>2\}.
\]
The induced action of $T/T_x$ on $M/T_x$ is a trivial principal
bundle, therefore Euler class of $T$-action on $M$ coincides with
the image of the Euler class of $T_x$-action on $M$ under the
inclusion map $i_x\colon T_x\to T$. The $T_x$-action on $M$ is the
Hopf bundle if $c_1,c_2$ have the same sign, and ``anti-Hopf''
bundle if $c_1,c_2$ have different signs. Its Euler class is
$\mu(F_{1,2})\in H^2(S^2_{1,2};H_1(T))$ in the first case and
$-\mu(F_{1,2})$ in the second case.
\end{proof}

\begin{figure}[h]
\begin{center}
\includegraphics[scale=0.3]{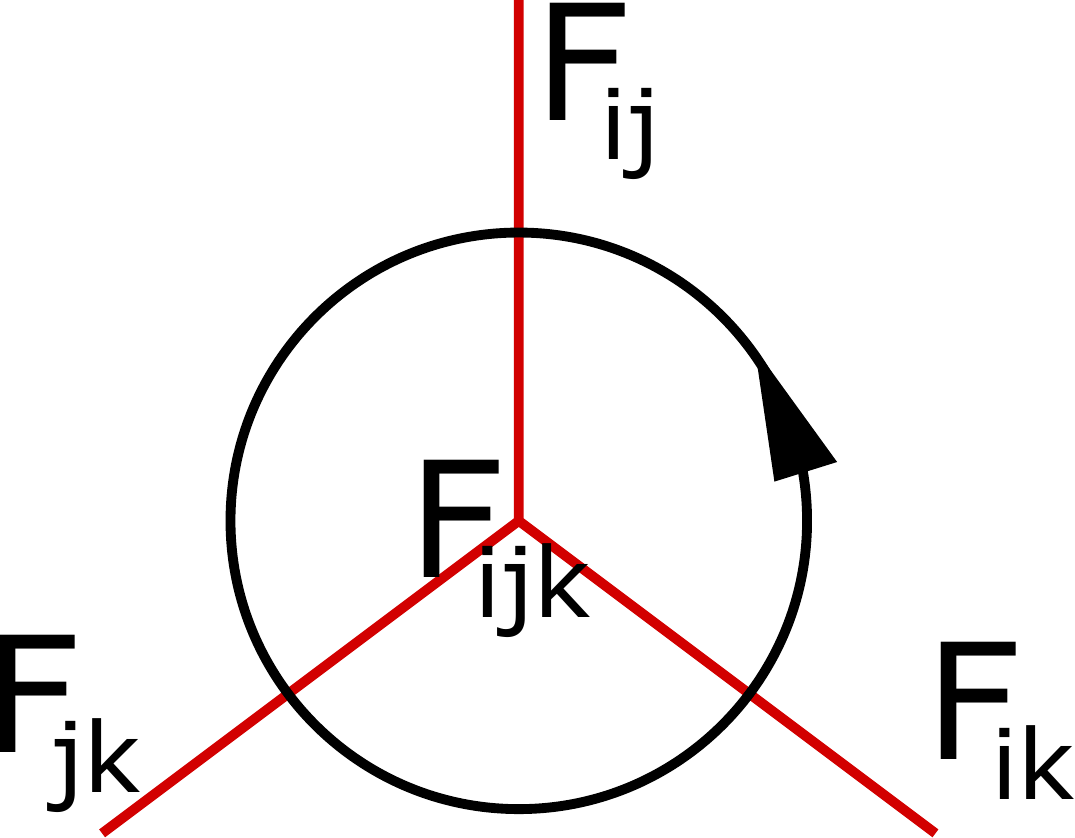}
\end{center}
\caption{Orienting three facets with a common
face}\label{figFacets}
\end{figure}

\begin{rem}
Note that there exist relations on the cycles $[S_{i,j}]\in
H_2(Q\setminus C;\Zo)\cong \Zo^{n-1}$. Every triple of indices
$i,j,k$ determines the $(n-3)$-face $F_{i,j,k}\in C$ which lies in
facets $F_{i,j},F_{j,k},F_{i,k}$. If we choose a small circle
around $F_{i,j,k}\subset \Ro^{n-1}$ and orient the facets
$F_{i,j},F_{j,k},F_{i,k}$ consistently (see schematic
Fig.\ref{figFacets}), we get a relation in $H_2(Q\setminus
C;\Zo)$:
\[
[S_{i,j}]+[S_{j,k}]+[S_{i,k}]=0.
\]
It implies the cocycle relation for stabilizers:
\begin{equation}\label{eqCocycleReln}
\frac{c_i}{c_j}\mu(F_{i,j})+\frac{c_j}{c_k}\mu(F_{j,k})+
\frac{c_i}{c_k}\mu(F_{i,k})=0.
\end{equation}
This relation is not surprising. Indeed, the product of the circle
subgroups $\mu(F_{i,j}),\mu(F_{j,k}),\mu(F_{i,k})$ inside the
torus $T$ have dimension $2$, therefore there should be exactly
one linear relation on their fundamental classes.
\end{rem}

Proposition \ref{propPMone} implies that for strictly appropriate
torus actions we have $e_x=\pm[\mu(F)]$ for $x\in F^\circ$, since
this holds in the local chart around fixed point.

\section{Reductions of locally standard
actions}\label{secLocStandActions}

A smooth manifold $X=X^{2n}$ with the action of $G=T^n$ is called
\emph{locally standard} if the action is locally modeled by the
standard representation of $G=T^n$ on $\Co^n$. Since
$\Co^n/T^n\cong \Rg^n$, the orbit space $P=X/G$ gets a natural
structure of a manifold with corners. Manifolds with locally
standard actions are classified up to equivariant homeomorphism
(see \cite{Yo}) by the following \emph{characteristic data}

\begin{enumerate}
\item The manifold with corners $P$, $\dim P=n$. There is a requirement that
every $k$-dimensional face of $Q$ is contained in precisely $n-k$
facets. Such manifolds with corners are called \emph{nice} in
\cite{MasPan} or just \emph{manifolds with faces}.

\item The characteristic function $\lambda$ which maps a facet $F$ of
$P$ into a circle subgroup of $G$: the stabilizer of any interior
point of $F$. Characteristic function satisfies the celebrated
(*)-condition: whenever facets $F_1,\ldots,F_k$ intersect in $P$,
the subgroups $\lambda(F_1),\ldots,\lambda(F_k)$ form a direct
product inside $G$. Since every circle subgroup of $G$ determines
a primitive integral vector in $\Hom(S^1,G)\cong \Zo^n$ up to
sign, it will be convenient to assume that $\lambda$ takes values
in the lattice.

\item The Euler class $e\in H^2(Q;H_1(G))\cong H^2(P^\circ;H_1(G))$,
which classifies the principal $G$-bundle $X^{\free}\to
X^{\free}/G=P^{\circ}$, where $X^{\free}$ is the free part of the
$G$-action.
\end{enumerate}

In the following we assume that every face of $P$ contains a
vertex so that the action is \emph{appropriate}.

A manifold $X$ with a locally standard action of $G$ is called a
\emph{quasitoric manifold} if $P=X/G$ is isomorphic to a simple
polytope as a manifold with corners. The free part of action is a
trivial $G$-bundle, since $P$ is contractible. So far, the Euler
class vanishes for quasitoric manifolds.

A fixed point $v$ of a locally standard action of $G$ on $X$
corresponds to a vertex $v$ of $P$ (we denote them by the same
letter). We have $v=F_1\cap\cdots\cap F_n$ for some facets
$F_i\subset Q$. Then the weights $\alpha_1,\ldots,\alpha_n\in
\Hom(G,S^1)=N$ of the tangent representation at $v$ is the dual
basis to $\lambda(F_1),\ldots,\lambda(F_n)\in \Hom(S^1,G)=N^*$.

Let $\{\alpha_{v,i}\}$ be a collection of all weights at all fixed
points. We can choose a generic homomorphism of the lattices
\[
\phi\colon \Hom(G,S^1)\cong\Zo^n\to \Zo^{n-1}
\]
such that the images
$\phi(\alpha_{v,1}),\ldots,\phi(\alpha_{v,n})\in \Zo^{n-1}$ are in
general position for any fixed point $v$. The homomorphism $\phi$
is determined by some homomorphism of tori $\phi^*\colon
T^{n-1}\to G$. Therefore the action of the subtorus
$T=\phi^*(T^{n-1})\subset G$ on $X$ is in general position.

\begin{thm}\label{thmQToricReduc}
Let $X=X^{2n}$ be a quasitoric manifold with the action of $G\cong
T^n$. Let $T\subset G$ be a subtorus of dimension $n-1$ such that
the induced action of $T$ on $X$ is an action in general position.
Then $X/T\cong S^{n+1}$.
\end{thm}

\begin{proof}
Denote the orbit space $X/T$ by $Q$ and the orbit space $X/G$ by
$P$. By the definition of a quasitoric manifold, $P$ is a simple
polytope, $\dim P=n$. We have a map $g\colon Q\to P$, which sends
a $T$-orbit to the $G$-orbit in which it lies. For any point $x$
in the interior of $P$ we have $g^{-1}(x)\cong S^1$. Since the
action is in general position, the preimage of a point $x\in\dd P$
is a single point (this fact was actually proved in Lemma
\ref{lemLocalQuotient} for a local chart). Since $P$ is
contractible, the map $g\colon Q\to P$ admits a section over
$P^\circ$. Therefore we have
\[
Q\cong P\times S^1/\sim
\]
where $\sim$ collapses circles over $\dd P$. Since $P$ is
homeomorphic to the $n$-disc $D^n$, we have
\[
Q\cong D^n\times S^1/\sim\cong \dd(D^n\times D^2)\cong S^{n+1},
\]
which proves the statement.
\end{proof}


We further investigate the characteristic data of the induced
action of $T\cong T^{n-1}$ on a quasitoric manifold. The arguments
in the proof of Theorem \ref{thmQToricReduc} imply the following
statement.

\begin{prop}
The sponge of the $T$-action on a quasitoric manifold $X$ has the
form
\[
Z\subset S^{n-1}\subset \Sigma^2S^{n-1}\cong Q,
\]
where $S^{n-1}$ is identified with the boundary of the polytope
$P$ and $Z$ is its $(n-2)$-skeleton. The facets of $Z$ are exactly
the faces of $P$ of codimension $2$.
\end{prop}

Note that characteristic function $\lambda$ of the $G$-action
determines the characteristic function $\mu$ of the $T$-action.
Let $F$ be a codimension-2 face of $P$ (hence a facet of $Z$).
Then $F=F_1\cap F_2$, where $F_1,F_2$ are the facets of $P$. We
have
\[
\mu(F)=\lambda(F_1)\times\lambda(F_2)\cap T
\]
Here $\lambda(F_1)\times\lambda(F_2)$ is a 2-torus in $G$, and
since $T$ is a codimension-1 subtorus of $G$ in general position,
the intersection $\lambda(F_1)\times\lambda(F_2)\cap T$ is a
1-dimensional subgroup, which is the stabilizer of the $T$-action
on interior of $F$. If we want this subgroup to be a circle
(recall that the definition of strictly appropriate action
requires that stabilizers don't have finite components), then the
subgroup $T\subset G$ is subject to some additional restrictions.
Namely, the subgroup $T\subset G$ determines the character
$\alpha_T\in \Hom(G,S^1)$, $\alpha_T\colon G\to G/T$. The next
lemma easily follows from Lemma \ref{lemPMone}:

\begin{lem}
The induced action of $T$ on a locally standard $G$-manifold $X$
is strictly appropriate if and only if
$\langle\alpha,\lambda(F_i)\rangle=\pm1$ for all facets $F_i$.
\end{lem}

\begin{ex}
Let $c\colon \{F_1,\ldots,F_m\}\to[n]$ be a proper coloring of
facets of a simple polytope $P$. This means, whenever $F_i$ and
$F_j$ are adjacent, their colors $c(F_i), c(F_j)$ are different.
Given such coloring we can construct a special characteristic
function $\lambda_c\colon\{F_1,\ldots,F_m\}\to\Zo^n$ which
associates to $F_i$ the basis vector $\lambda(F_i)=
\epsilon_{c(F_i)}\in \Zo^n$. Such characteristic functions and
corresponding quasitoric manifolds $X_{(P,\lambda_c)}$ were called
\emph{pullbacks from linear model} in \cite{DJ}. It can be seen
that the induced action of the subtorus
\[
T=\{t_1^{c_1}t_2^{c_2}\cdots t_n^{c_n}=1\}\subset G \qquad
c_i=\pm1,
\]
on $X_{(P,\lambda_c)}$ is strictly appropriate.

Note that there exist other examples of strictly appropriate
induced actions which do not come from colored characteristic
functions.
\end{ex}

%

The Euler class $e$ of the induced action of $T$ on a quasitoric
manifold $X$ determines the action.

\begin{thm}\label{thmClassSphere}
Let $X'$ and $X''$ be two manifolds with strictly appropriate
actions in general position. Let $(Q'\cong S^{n+1},Z',\mu',e')$
and $(Q''\cong S^{n+1},Z'',\mu'',e'')$ be their characteristic
data. Suppose there is a homeomorphism of pairs $(Q',Z')\cong
(Q'',Z'')$ and $e'_x=e''_x$ for any point $x$ in a sponge. Then
$X'$ and $X''$ are equivariantly homeomorphic.
\end{thm}

\begin{proof}
Taking $x$ in the interior of a facet $F$ of a
sponge $Z'\cong Z''$, we see that $\mu'(F)=\mu''(F)$ since $e_x'$
is the fundamental class of $\mu'(F)$ and $e_x''$ is the
fundamental class of $\mu''(F)$. Hence $\mu'=\mu''$.

Let $(Q,Z)$ be either $(Q',Z')$ or $(Q'',Z'')$ and let
$U=\bigcup_{x\in Z}U_x$ be a neighborhood of $Z$ in $Q$. As
before, $U_x$ denotes a small neighborhood of $x\in Z$
homeomorphic to an open ball. The local classes $e_x$ determine
the classes $e_x'\in H^3(U_x,U_x\setminus Z;\Zo^{n-1})$ due to the
exact sequence
\[
\xymatrix{ H^2(U_x;\Zo^{n-1})\ar[r] \ar@{=}[d]& H^2(U_x\setminus
Z;\Zo^{n-1}) \ar[r] &
H^3(U_x,U_x\setminus Z;\Zo^{n-1}) \ar[r]& H^3(U_x;\Zo^{n-1})\ar@{=}[d]\\
0& e_x\ar@{|->}[r]& e_x'&0}
\]
The classes $\{e_x'\mid x\in Z\}$ determine a unique element
$e'\in H^3(U,U\setminus Z;\Zo^{n-1})$ such that $i_x^*(e')=e'_x$
for an inclusion $i_x\colon U_x\hookrightarrow U$ according to
Mayer--Vietoris argument. By excision, we can view $e'$ as an
element in $H^3(Q,Q\setminus Z;\Zo^{n-1})\cong H^3(U,U\setminus
Z;\Zo^{n-1})$. Recall that $Q\cong S^{n+1}$. From the exact
sequence
\[
\xymatrix{ H^2(Q;\Zo^{n-1})\ar[r] \ar@{=}[d]& H^2(Q\setminus
Z;\Zo^{n-1}) \ar[r] &
H^3(Q,Q\setminus Z;\Zo^{n-1}) \ar[r]& H^3(Q;\Zo^{n-1})\ar@{=}[d]\\
0& e\ar@{|->}[r]& e'&0}
\]
we extract a unique element $e\in H^2(Q\setminus Z;\Zo^{n-1})$
which projects to $e_x$ for any point $x\in Z$.

Characteristic data $(Q'\cong S^{n+1},Z',\mu',e')$ and $(Q''\cong
S^{n+1},Z'',\mu'',e'')$ coincide, hence the spaces $X'$ and $X''$
are equivariantly homeomorphic to the model space according to
Proposition \ref{propHomeoToModel}. Thus they are homeomorphic to
each other.
\end{proof}

\begin{rem}
Instead of the equality $e_x'=e_x''$ one can require the equality
of characteristic functions $\mu'=\mu''$, and for a small 2-sphere
around each facet $F$ specify the type of its preimage (whether it
is Hopf or anti-Hopf bundle, see Proposition \ref{propPMone}). If
the types agree for $X$ and $X'$ then equality $\mu'=\mu''$ would
imply equality of local classes $e_x'=e_x''$.
\end{rem}

In order to study certain examples, we need a modification of
Theorem \ref{thmClassSphere}. Let $M$ be a closed manifold of
dimension $n-1$. Assume there is a regular simple cell subdivision
on $M$ which means there is a given regular cell structure in
which every $k$-dimensional cell is contained in exactly $n-k$
maximal cells. Its $(n-2)$-skeleton $Z_M=M_{n-2}$ is a sponge.
Consider the manifold with boundary $Q_M=M\times D^2$. We consider
$M$ as a subset $M\times{0}\subset Q_M$.

\begin{prop}\label{propSpecialHomeo}
Let $(X,\dd X)$ be a $2n$-dimensional manifold with boundary, and
assume there is an appropriate action of $T=T^{n-1}$ on $X$ with
characteristic data $(Q_M,Z_M,\mu_M,e_M)$. We also assume that the
action is free on the boundary $\dd X$ and the principal
$T$-bundle $\dd X\to \dd X/T=\dd Q_M\cong M\times \dd D^2$ is
trivial. Then the class $e_M\in H^2(Q_m\setminus Z_M;\Zo^{n-1})$
is uniquely determined by the local classes $e_x$, $x\in Z_M$.
\end{prop}

\begin{proof}
There is an exact sequence of the pair $(Q_M\setminus Z_m,\dd
Q_M)$:
\[
\xymatrix{ H^2(Q_M\setminus Z_M,\dd Q_M;\Zo^{n-1})\ar[r] &
H^2(Q_M\setminus Z_M;\Zo^{n-1}) \ar[r] &
H^2(\dd Q_M;\Zo^{n-1}) 
}
\]
The class $e\in H^2(Q_M\setminus Z_M;\Zo^{n-1})$ maps to zero
since the free part of action is a trivial $T$-bundle over $\dd
Q$. Hence there exists $\tilde{e}\in H^2(Q_M\setminus Z_M,\dd
Q_M;\Zo^{n-1})$ which maps to $e$, and $e$ is uniquely determined
by the class $\tilde{e}$. We have
\[
(Q_M\setminus Z_M)/\dd Q_M\simeq \Sigma^2(M\setminus Z_M)
\]
hence $H^2(Q_M\setminus Z_M,\dd Q_M;\Zo^{n-1})\cong
\Hr^0(M\setminus Z_M)$. The space $M\setminus Z_M$ is the disjoint
union of open top-dimensional cells of $M$. It can be seen that
cohomology classes of $H^2(Q_M\setminus Z_M,\dd Q_M;\Zo^{n-1})$
are localized near $Z_M$ thus are completely determined by the
local classes.
\end{proof}

\begin{cor}
Under the assumptions of Proposition \ref{propSpecialHomeo}, the
equivariant homeomorphism type of $X$ is determined by $(Q_M,Z_M)$
and the weights of all tangent representations at all fixed
points.
\end{cor}

\begin{con}\label{constrDiscTimesMfd}
The examples of the actions above can be constructed in the
following way. We consider a manifold $P\cong M\times [0;1]$ with
boundary $\dd P=\dd_0P\sqcup\dd_1P$, $\dd_iP=M\times\{i\}$, and
endow it with the structure of a nice manifold with corners.
Namely, we subdivide the boundary component $\dd_0P$ according to
the subdivision of $M$ and do nothing with $\dd_1P$ (this boundary
component is considered a single face of dimension $n-1$). Now we
may take an abstract characteristic function satisfying
$(*)$-condition:
\[
\lambda\colon \{\mbox{facets of }\dd_0P\}\to \Hom(S^1,G)\cong
\Zo^n,
\]
and construct a topological manifold
\[
X=(P\times G)/\sim
\]
with boundary $\dd X=\dd_1P\times G$. Here $G\cong T^n$ and $\sim$
collapses tori over $\dd_0P$ according to characteristic function
(refer to \cite{DJ,Yo,BPnew} for details). These particular
manifolds with boundary were studied in \cite{AyzMexicana}.

We take a generic $(n-1)$-dimensional subtorus $T\subset G$ so
that the induced action of $T$ on $X$ is strictly appropriate and
in general position. It can be seen that the orbit space $Q=X/T$
is homeomorphic to
\[
Q\cong P\times S^1/\sim = (M\times[0;1])\times S^1/\sim
\]
where the circles collapse over $\dd_0P=M\times\{0\}$. Therefore
$Q\cong M\times D^2$. The sponge of the $T$-action is the
$(n-2)$-skeleton of $M=M\times \{0\}\subset M\times D^2$. Finally,
the free $T$-action over $M\times\dd D^2$ is a trivial bundle,
since the $G$-action is trivial over $\dd_1P$.
\end{con}

\section{Grassmann and flag manifolds}

Next we review two classical examples motivating our study.

\begin{ex}
The standard action of a compact torus $T^4$ on $\Co^4$ induces
the action of $T^4$ on a Grassmann manifold $G_{4,2}$ of complex
$2$-planes in $\Co^4$. This action has non-effective kernel
$\Delta(T^1)\subset T^4$, hence we have an effective action of
$T=T^4/\Delta(T^1)\cong T^3$ on $G_{4,2}$, $\dim_\Ro G_{4,2}=8$.
There are $6$ fixed points, and it is not difficult to find the
weights of their tangent representations. The easiest way to do
this is to look at the image of the moment map, which coincides
with a regular octahedron $\Delta_{4,2}$. Its vertices correspond
to the fixed points, and the primitive lattice vectors along the
edges of octahedron correspond to the weights of the tangent
representation. For example, the edges from the top vertex
$(0,0,1)$ of octahedron are
\[
\alpha_1=(1,0,-1),\quad \alpha_2=(0,1,-1),\quad
\alpha_3=(-1,0,-1),\quad \alpha_4=(0,-1,-1).
\]
Every $3$ of them are linearly independent, hence the action is in
general position. The action is strictly appropriate.

It was proved in \cite{BT}, that the orbit space $G_{4,2}/T$ is
homeomorphic to $S^5$. The sponge $Z$ of the action is obtained by
taking the boundary of octahedron $\dd\Delta_{4,2}$, and attaching
three squares along the equatorial cycles as shown on
Fig.\ref{figOctah}.

\begin{figure}[h]
\begin{center}
\includegraphics[scale=0.3]{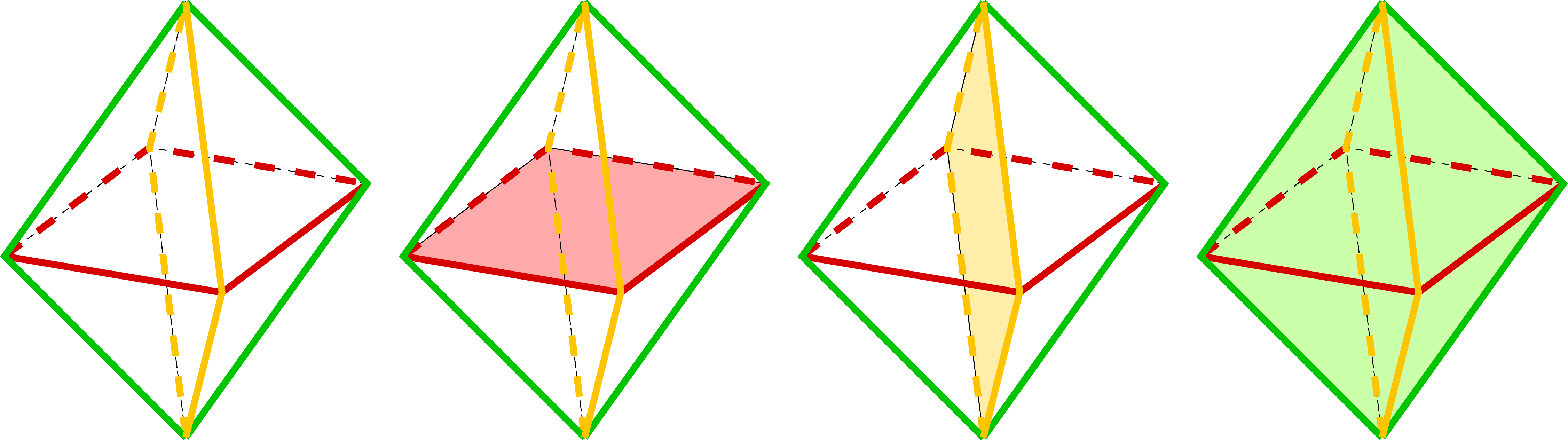}
\end{center}
\caption{The sponge of $G_{4,2}$ consists of the boundary of
octahedron with 3 squares attached along the
equators}\label{figOctah}
\end{figure}
\end{ex}

\begin{ex}
The standard action of $T^3$ on $\Co^3$ induces the effective
action of $T=T^3/\Delta(T^1)$ on the manifold $F_3$ complete
complex flags in $\Co^3$. We have $\dim T=2$, $\dim F_3=6$. There
are $6$ fixed points and the tangent representation at each point
is in general position. The action has no finite components in
stabilizers.

\begin{figure}[h]
\begin{center}
\includegraphics[scale=0.2]{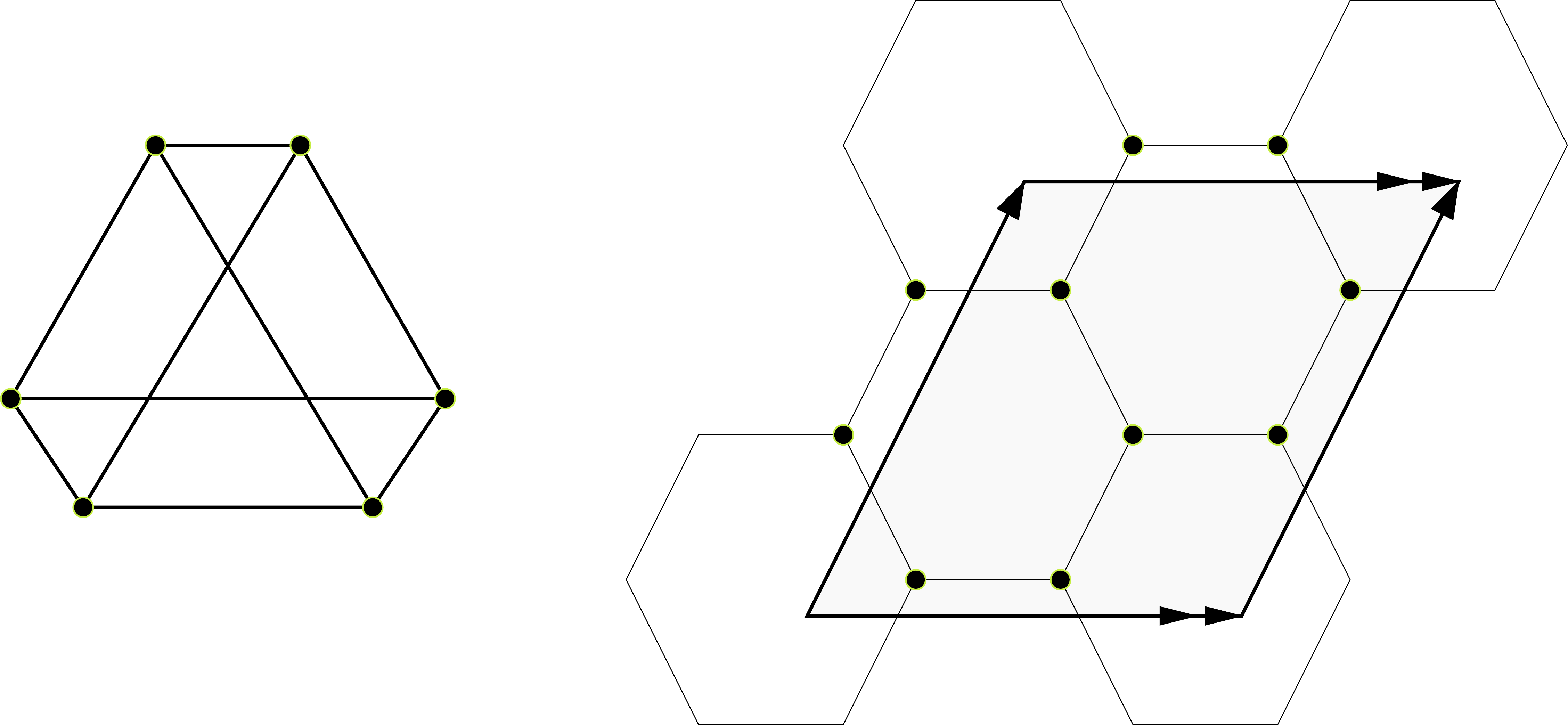}
\end{center}
\caption{The sponge of complete flag manifold
$F_3$}\label{figFlag}
\end{figure}

Using the technique of \cite{BT} (see \cite{AyzMatr} for
alternative proof) it can be shown that the orbit space $F_3/T$ is
homeomorphic to $S^4$. The sponge of the action has dimension $1$.
This is simply the GKM-graph of the action, which is well known.
This graph is shown on Fig.\ref{figFlag}. As an abstract graph it
is a complete bipartite graph $K_{3,3}$. The figure on the right
shows how to realize this graph as a 1-skeleton of simple cell
structure on a 2-torus $\T$. Actually, $\T$ can be embedded in
$S^4=F_3/T$ in a canonical way and the preimage of its small
neighborhood $U_\T$ under the projection map is described by
Construction \ref{constrDiscTimesMfd}. This subject will be
covered in detail in a subsequent paper \cite{AyzMatr}.
\end{ex}

Note the geometrical difference of these two examples from the
induced $T$-action on a quasitoric manifold. In case of $T$-action
on a quasitoric manifold, the sponge, which is an
$(n-2)$-dimensional complex, can be embedded in $S^{n-1}$ (since
it is the $(n-2)$-skeleton of a polytope). However the sponges of
$G_{4,2}$ and $F_3$ do not embed in a sphere as codimension one
complexes. In case of $F_3$ the graph $K_{3,3}$ is well-known to
be non-planar. The sponge of $G_{4,2}$, which is the octahedron
with $3$ squares attached, cannot be embedded in $\Ro^3$.

\begin{rem}
Whenever the orbit space $Q=X/T$ is a sphere $S^{n+1}$, Alexander
duality implies $H^2(Q\setminus Z; R)\cong H_{n-2}(Z;R)$ for a
sponge $Z\subset Q$. The homology class corresponding to $e\in
H^2(Q\setminus Z;H_1(T))$ is represented by the chain
\[
\sigma=\sum_{F :\mbox{ facet of }Z}e_x\cdot[F]\in
C_{n-2}(Z;H_1(T)).
\]
Here $[F]$ is the fundamental class of a facet $F$ and $e_x\in
H_2(U_x;H_1(T))\cong H_1(T)$ is the local Euler class in an
interior point $x\in F^\circ$. The chain $\sigma$ is a cycle
according to relation \eqref{eqCocycleReln}.
\end{rem}

\section{Acknowledgements}

The author deeply appreciates the help of Victor Buchstaber,
especially his advices and explanation concerning the theory of
$(2n,k)$-manifolds. The author thanks Mikiya Masuda for a very
geometrical explanation of Bialynicki--Birula theory.

\end{document}